\documentclass[12pt,leqno]{amsart}
\usepackage{amssymb,amsmath,amsthm,amsxtra}
\usepackage[textsize=tiny]{todonotes}

\usepackage{graphicx}
\usepackage{color}
\usepackage{xcolor,import}

 \definecolor{darkblue}{RGB}{0,0,160}

\usepackage[lining]{libertine}   
\usepackage{cabin}
\usepackage[libertine]{newtxmath}
\usepackage[colorlinks=true,allcolors=darkblue]{hyperref}
\usepackage[backrefs,msc-links,nobysame,initials,non-compressed-cites]{amsrefs}

\usepackage[T1]{fontenc}

\usepackage{color}

\def\e{{\rm e}}
\def\eps{\varepsilon}

\def\d{{\rm d}}

\def\R {\mathbb{R}}

\def\v {{\mathbf \Theta}}

\def\J {{\mathbb J}}

\def\M {{\mathrm M}}

\def\1 {{\mbox{\boldmath 1}}}

\def \and{\quad\text{and}\quad}

\newcommand{\cic}[1]{\mbox{\boldmath$#1$}}
\newcommand{\Exp}[0]{\mathbb{E}}

\def \no#1#2#3 {{\bf #1} (#3), #2.}
\def \eds#1#2#3 {#1, #2, #3.}
\newcounter{counter2}
\numberwithin{equation}{section}
\numberwithin{counter2}{section}
\newtheorem{proposition}[counter2]{Proposition}

\newtheorem{theorem}[counter]{Theorem}

\newtheorem{corollary}[counter2]{Corollary}
\newtheorem{lemma}[counter2]{Lemma} 
 
\theoremstyle{definition}
\newtheorem{definition}{Definition}
\newtheorem*{remark*}{Remark}
\newtheorem*{warn*}{A word of warning}

\newtheorem{remark}[counter2]{Remark} 
\theoremstyle{plain}

\def\XXint#1#2#3{{\setbox0=\hbox{$#1{#2#3}{\int}$}
\vcenter{\hbox{$#2#3$}}\kern-.5\wd0}}

\begin{document}
%%%%%%%%%%%%%%%%%%%%%%%%%%%%%%%%%%%%%%%%%%%%%%%%%

\title[Finite order lacunary Hilbert transform]{A sharp estimate for the Hilbert transform along finite order lacunary sets of directions}

\author[F. Di Plinio]{Francesco Di Plinio} \address{\noindent Department of Mathematics, University of Virginia, Box 400137, Charlottesville, VA 22904, USA}
\email{\href{mailto:francesco.diplinio@virginia.edu}{\textnormal{francesco.diplinio@virginia.edu}}}
\thanks{F. Di Plinio was partially supported by the National Science Foundation under the grants NSF-DMS-1500449 and  NSF-DMS-1650810, by the Severo Ochoa Program SEV-2013-0323 and by Basque Government BERC Program 2014-2017}

\author[I. Parissis]{Ioannis Parissis}
\address{Departamento de Matem\'aticas, Universidad del Pais Vasco, Aptdo. 644, 48080 Bilbao, Spain and Ikerbasque, Basque Foundation for Science, Bilbao, Spain}

\email{\href{mailto:ioannis.parissis@ehu.es}{\textnormal{ioannis.parissis@ehu.es}}}
\thanks{I. Parissis is supported by grant  MTM2014-53850 of the Ministerio de Econom\'ia y Competitividad (Spain), grant IT-641-13 of the Basque Government, and IKERBASQUE}

\subjclass[2010]{Primary: 42B20. Secondary: 42B25}
\keywords{Directional operators, lacunary sets of finite order, Stein's conjecture}

%%%%%%%%%%%%%%%%%%%%%%%%%%%%%% ABSTRACT ABSTRACT ABSTRACT
\begin{abstract}Let $\v\subset S^1$ be a lacunary set of directions of order $D$. We show that the maximal directional Hilbert transform
\[
H_{\v} f(x)\coloneqq \sup_{v\in \v} \Big|\mathrm{p.v.}\int_{\R }f(x+tv)\frac{\d t}{t}\Big|
\]
obeys the bounds $\|H_{\v}\|_{L^p\to L^p}\simeq_{p,D} (\log\#\v)^\frac{1}{2}$, for all $1<p<\infty$.  For vector fields $\mathsf{v}_D$ with range in a lacunary set of of order $D$ and generated using suitable combinations of truncations of Lipschitz functions, we prove that the truncated Hilbert transform along the vector field $\mathsf{v}_D$,
\[
H_{\mathsf{v}_D,1} f(x)\coloneqq \mathrm{p.v.} \int_{  |t| \leq 1 } f(x+t\mathsf{v}_D(x)) \,\frac{\d t}{t}, 
\]
satisfies the bounds $\|H_{\mathsf {v}_D,1}\|_{L^p\to L^p}\lesssim_{p,D} 1$ for all $1<p<\infty$. These results extend previous bounds of the first author with Demeter, and of Guo and Thiele.
\end{abstract}
%%%%%%%%%%%%%%%%%%%%%%%%%%%%%% ABSTRACT ABSTRACT ABSTRACT

\maketitle

%%%%%%%%%%%%%%%%%%%%%%%%%%%%%% SECTION SECTION SECTION
\section{Introduction} Our subject is Hilbert transforms, and more general singular integrals and maximal averages along directions in the plane. The simplest example of such an operator is described by fixing a nonzero vector $v \in \R^2$ and considering the multiplier 
\begin{equation}\label{e.Hilb+}
H_v ^+f(x)\coloneqq  \int_{\R^2} \widehat f(\xi) \cic{1}_{[0,\infty)}(\xi\cdot v) \e^{ix\cdot \xi} \, \d \xi,\qquad x\in\R^2,
\end{equation}
acting initially on Schwartz functions $f$ defined on $\R^2$. Up to a linear combination with the identity operator, this operator coincides with the Hilbert transform along the direction $v$
\[
H_v f(x)\coloneqq  \mathrm{p.v.}\int_{\R} f(x+tv) \,\frac{\d t}{t}. 
\] 
Notice that $H_{v}$ is dilation invariant and thus we can always replace $v$ by its projection on $S^1$, which we identify with $v$ when $v$ is acting as a parameter for the definition of $H_v$.

The companion maximal operator is the maximal average of a function $f$ in the direction given by $v$
\[
\M_v f(x)\coloneqq \sup_{\eps>0}\frac{1}{2\eps}\int_{-\eps} ^\eps |f(x+tv)|\, \d t,\qquad x\in \R^2.
\]
Such operators arise for example if one applies the method of rotations to singular integrals in $\R^2$ which are given by $-2$-homogeneous kernels with mean zero on $S^1$. For a fixed direction $v\in S^1$, one-dimensional theory implies that $H_v,\M_v$ are bounded operators on $L^p(\R^2)$, $1<p<\infty$, with bounds independent of $v$. Of course the maximal operator $\M_v$ is also trivially bounded on $L^\infty(\R^2)$. See for example \cite[\S4.3]{Duo} for a relevant discussion.

Things become more interesting when one seeks for bounds for the corresponding maximal versions of these operators along sets of directions. To make this more precise, let $\v \subset S^1$ and consider the operators
\[
\M_\v f(x)\coloneqq \sup_{v\in \v} \M_{v}f(x),\quad H_\v f(x)\coloneqq \sup_{v\in\v}|H_{v}f(x)|,\qquad x\in \R^2.
\]
We naturally ask under which conditions on $\v$ these operators are bounded on $L^p(\R^2)$, at least for some $p\in (1,\infty)$. In order to discuss the answer we need the definition of lacunary sets.

%%%%%%%%%%%%%%%%%%%%%%%%%%%%%% SECTION SECTION SECTION
\subsubsection*{Lacunary sets of finite order } Following \cite{SS} we give below the definition of a $D$-lacunary set. To that end we first define the notion of a \emph{successor} set. Throughout the following definitions, $\lambda\in (0,1)$ is a fixed parameter.

%%%%%%%%%%%%%%%%%%%%%%%%%%%%%% DEFINITION DEFINITION DEFINITION
\begin{definition}[successor] In what follows,  $'\v,\v\subset S^1$ are closed sets of measure $0$. For $x,y\in S^1$ denote by  $\mathrm{dist}(x,y)$  the (geodesic) distance between $x$ and $y$, and by $\mathrm{dist}_{'\v}(y)$   the (geodesic) distance of $y$ from $'\v$. Then the set $\v$ is called a \emph{successor} of $'\v$  if for all $x,y\in\v$ with $x\neq y$ we have that $\mathrm{dist}(x,y)\geq (1/\lambda-1) \mathrm{dist}_{'\v}(x)$.
\end{definition}
%%%%%%%%%%%%%%%%%%%%%%%%%%%%%% DEFINITION DEFINITION DEFINITION

With this definition in hand we can describe lacunary sets of finite order.

%%%%%%%%%%%%%%%%%%%%%%%%%%%%%% DEFINITION DEFINITION DEFINITION
\begin{definition}[$D$-lacunary sets] A $0$-lacunary set is a single point set $\v=\{v_\infty\}$ for some $v_\infty \in S^1$. 
	
	A set $\v\in S^1$ is called \emph{clockwise lacunary} if $\v=\{v_k\}_{k=1} ^\infty$, all the directions in $\v$ are ordered in a clockwise fashion, and there exists a direction $v_\infty\in S^1$ and $0<\lambda<1$ such that for all $k\geq 1$ we have $v_k\neq v_\infty$ and
\[
\mathrm{dist}(v_{k+1},v_\infty)<\lambda \mathrm{dist}(v_k,v_\infty).
\]
We call $v_\infty$ the limit of $\v$. The definition of a \emph{counterclockwise lacunary} set $\v$ is completely analogous. A set $\v\subset S^1$ is called $1$-lacunary (or, simply, lacunary) if it can be written as the union of a clockwise lacunary set and a counterclockwise lacunary set with the same limit $v_\infty$ (allowing for one of the two sets in the union to be the empty set).
	
Given a positive integer $D\geq 1$ we say that $\v$ is $D$-lacunary set if $\v$ is a successor of a $(D-1)$-lacunary set $'\v$. It is important to note that the value of the successor constant $\lambda>0$ remains fixed throughout the inductive definition. Note that, with this definition, every $D$-lacunary set $\v$ is associated with a single element $v_\infty\in S^1$ that will be called \emph{the root} of $\v$.
\end{definition}
%%%%%%%%%%%%%%%%%%%%%%%%%%%%%% DEFINITION DEFINITION DEFINITION

Some remarks concerning the definition above are in order. Supposing that $\v$ is a $D$-lacunary set, then there exists a $(D-1)$-lacunary set $'\v$ and $\v$	is the successor of $'\v$. Let $'I_k$ denote the complementary arcs of $'\v$ in $S^1$. Then for each $k$ the set $\v\cap {'I_k}$ is contained in the union of a clockwise lacunary sequence converging to the right (as, clockwise) endpoint of $'I_k$ and a counterclockwise lacunary sequence converging to the left (as, counterclockwise) endpoint of $'I_k$. We will repeatedly use this structure of $D$-lacunary sets in what follows.

The canonical example of a clockwise lacunary set of directions with limit $0$ is the set $\{\pi 2^{-j}\}_{j\geq 1}$. On the other hand, the set
\[
\{\pi (2^{-j}+2^{-k})\}_{j\geq 1,k>j}
\]
is a $2$-lacunary set. 

With the definition of a $D$-lacunary set in hand, we now return to the questions of boundedness of $\M_\v$ and $H_\v$. If $\v$ is a $D$-lacunary set then $\M_v$ is bounded on $L^p(\R^2)$ for all $1<p<\infty$. For $D=1$ this is the main result of \cite{NSW}. This result was extended to lacunary sets of finite order in \cite{SS}. The converse implication is a striking result of Bateman, \cite{Bat}: if $\M_v$ is bounded on $L^p(\R^2)$ for some $p\in (1,\infty)$ then $\v$ is a $D$-lacunary set for some nonnegative integer $D$. In \cite{Katz}, Katz proved the quantitative sharp bounds for arbitrary finite sets of directions $\v$ 
\[
\|\M_\v\|_{L^2\to L^2}\lesssim \log \#\v, \quad \|\M_\v\|_{L^2\to L^{2,\infty}}\lesssim (\log \#\v)^\frac{1}{2},\quad\|\M_\v\|_{L^p\to L^p}\lesssim (\log \#\v)^\frac{1}{p},\quad p>2.
\]
Note that the result from \cite{SS} implies that $\|\M_\v\|_{L^p\to L^p} \lesssim_D 1$ whenever $\v$ is a $D$-lacunary set of directions.

For the Hilbert transform $H_\v$, Karagulyan showed in \cite{Karag} that for all set $\v\subset S^1$ we have
\[
\|H_\v \|_{L^2(\R^2)} \gtrsim  (\log{\#\v})^\frac{1}{2}.
\]
In particular, the Hilbert transform $H_{\v}$ is unbounded whenever $\v$ is an infinite set of directions. 

Concerning upper bounds for $H_\v$, keeping in mind the above result of Karagulyan, we only consider $\v\subset S^1$ with $\#\v<\infty$. Then the main result is due to the first author and Demeter \cite{DDP}: if $\v$ is any finite set of directions then $\|H_\v\|_{L^p\to L^p}\lesssim_{p} \log\#\v$ for all $2\leq p<\infty$.  The case $p=2$  of these norm bounds  was previously treated in \cite{Dem}. In \cite{DDP} it is also shown that under additional structural assumptions on $\v$, including for instance the equi-spaced case $\v=\{\e^{2\pi i \frac k N}: k=1,\ldots, N\}$, the quantitative estimate  improves to $\|H_\v\|_{L^p\to L^p}\lesssim_{p,\eps} (\log\#\v)^{\frac12+\eps}$ for $2<p<2+\eps$ and $\eps>0$ sufficiently small: the proof uses product-BMO type phase plane analysis.

In the same paper \cite{DDP} the authors prove that if $\v$ is a $1$-lacunary set then the upper bound improves to $\|H_\v\|_{L^p\to L^p}\lesssim_p  (\log\#\v)^\frac{1}{2}$ for all $1<p<\infty$. The first main result of this paper is the  $D$-lacunary version of the upper bounds for $H_\v$.
%%%%%%%%%%%%%%%%%%%%%%%%%%%%%% THEOREM THEOREM THEOREM
\begin{theorem} \label{t.D-lacunary:main} Let $\v\subset S^1$ be a $D$-lacunary set of directions.  The maximal directional Hilbert transform  $H_\v f(x)\coloneqq \sup_{v \in \v} |H_{v} f(x)|$
obeys  the bounds
\begin{equation} \label{t.D-lacunary:mainbd} 
c (\log \# \v)^{\frac12} \leq \|H_\v \|_{L^p(\R^2)} \leq C (\log \# \v)^{\frac12},\qquad  1 < p < \infty,
\end{equation}
with constants $c,C>0 $ depending only on $D,p$. 
\end{theorem}
%%%%%%%%%%%%%%%%%%%%%%%%%%%%%% THEOREM THEOREM THEOREM

The maximally truncated Hilbert transform in the direction $v\in S^1$ is 
\[
H_{v} ^*f(x)\coloneqq \sup_{\eps>0}\bigg|\int_{|t|>\eps} f(x-tv)\, \frac{\d t}{t}\bigg|
\]
and then we set $H_\v ^*f(x)\coloneqq \sup_{v\in\v}H_{v} ^*f(x)$. The previous theorem together with the one-dimensional Cotlar inequality for the maximal truncations of the Hilbert transform, see for example \cite[Lemma 3.5]{Duo},
\[
H_{v} ^* f(x) \leq \M_{v}(H_{v}f)(x)+C \M_{v}f(x),
\]
implies the corresponding estimate for $H_\v ^*$
\[
H_\v ^*f (x) \leq \M_\v(H_\v f)(x)+C\M_\v f(x).
\]
This together with the boundedness of $\M_\v$ on $L^p(\R^2)$ for all $p>1$ immediately imply the following corollary for the maximal truncations of $H_\v$.
%%%%%%%%%%%%%%%%%%%%%%%%%%%%%% COROLLARY COROLLARY COROLLARY
\begin{corollary}\label{c.truncated} Let $\v\subset S^1$ be a $D$-lacunary set of directions. We have 
\[ 
 \|H_\v ^* \|_{L^p(\R^2)} \leq C (\log \# \v)^{\frac12},  \qquad 1<p<\infty,
\]
where $C>0$ depends only on $D,p$.
\end{corollary}
%%%%%%%%%%%%%%%%%%%%%%%%%%%%%% COROLLARY COROLLARY COROLLARY

The results and setup described above naturally connect to the Hilbert transform along vector fields. To make this precise, let $\v$ be a set of directions and consider the operator
\[
H_\v f(x)\coloneqq \sup_{v \in\v}\left| \mathrm{p.v.}\int_{\R} f(x+tv )\frac{\d t}{t}\right|.
\]
Then there exists a measurable function $\mathsf{v}:\R^2\to \v$ such that, up to a sign, $H_\v $ is pointwise equivalent to
\[
f\mapsto  \mathrm{p.v.}\int_{\R} f(x+t\mathsf{v}(x))\frac{\d t}{t}.
\]
As described in \cite{LacLi:tams}, there is a smooth vector field $\mathsf v$ such that the untruncated version considered above fails to be $L^2$-bounded even after precomposition with a smooth   restriction to a frequency annulus.
Such obstruction, caused  by    the large scales of the kernel, is obviated  naturally by considering instead the operator
\[
H_{\mathsf{v},\eps} f(x)\coloneqq\mathrm{p.v.} \int_{ |t| \leq \eps} f(x+t\mathsf{v}(x)) \,\frac{\d t}{t}  
\]
where $\mathsf{v}:\R^2\to \R^2$ is a sufficiently smooth vector field. The truncation at scale $\eps$ should be chosen in inverse relation to some modulus of smoothness of $\mathsf{v}$. Similarly, one is led to consider the truncated maximal average of a Schwartz function $f$ along the vector field $\mathsf{v}$ 
\[
\M_{\mathsf{v},\eps} f(x)\coloneqq\sup_{0<s\leq \epsilon} \frac{1}{2s}\int_{ |t| \leq s} |f(x+t\mathsf{v}(x))| \,\d t .
\]
The study of this operator is more naturally motivated by the question of differentiation of functions along vector fields, and then the truncation to small scales is immediately justified. 

In the context described above, Zygmund conjectured     the weak-$L^2$ boundedness of the maximal function $\M_{\mathsf{v},\eps}$  whenever $\mathsf{v}$ is a Lipschitz vector field and $\eps^{-1}\simeq \|\mathsf{v}\|_{\mathsf{LIP}}$. The analogous  boundedness property for  the Hilbert transform $H_{\mathsf{v},\eps}$ under the same assumptions on $\mathsf{v}$ and $\eps$ was later conjectured by Stein \cite{STP}. These conjectures have been   answered in the affirmative in the case that $\mathsf{v}$ is an \emph{analytic} vector field: by Bourgain, \cite{Bourgain}, for the maximal operator  and by Stein and Street, \cite{StStr}, for the Hilbert transform. These results were later recovered and extended by Lacey and Li to $C^{1+\eta}$-vector fields, $\eta>0$, satisfying an additional geometric condition which holds when $\mathsf{v}$ is analytic; see \cites{LacLi:mem,LacLi:tams}. For further refinements and additional results we point the interested reader to \cites{Bat1v,BatThiele,Guo1,GuoThiele,Guo} and to the references therein.

The second main result of the paper, Theorem~\ref{t.Liplac} below, extends the class of vector fields   along which the truncated Hilbert transform admits $L^p$-bounds to include those    with $D$-la\-cu\-na\-ry range and satisfying an additional assumption of Lipschitz flavor.  Our result extends \cite[Theorem 1.2]{GuoThiele} by Guo and Thiele to lacunary orders $D>1$. In the statement, and in what follows, we denote by $\lfloor x \rfloor$ the largest integer which is less than or equal to $x$.

%%%%%%%%%%%%%%%%%%%%%%%%%%%%%% THEOREM THEOREM THEOREM
\begin{theorem}\label{t.Liplac} Let $D$ be a positive integer. For $j=1,\ldots,D$ let $\lambda_j:\R^2\to (0,1]$ be Lipschitz functions with $\|\lambda_j\|_{\mathsf{LIP}}\leq 1$ and with the additional property that
\[
\lambda_{j}(x) \leq 2^{-5} \lambda_{j-1}(x), \qquad \forall x\in \R^2, \quad j=2,\ldots,D.
\]
Define $\mathsf{v}_D:\R^2\to S^1$ by
\begin{equation} \label{vD}
\mathsf{v}_D(x)\coloneqq  \prod_{j=1}^D \e^{2\pi i 2^{\lfloor \log_2 \lambda_j(x)\rfloor}}.
\end{equation}
Then the truncated Hilbert transform $H_{\mathsf{v}_D,1}$ satisfies the norm inequality
\begin{equation}\label{estimatemain1} 
	\|H_{\mathsf{v}_D,1} f\|_{L^p(\R^2)} \leq C  \|f\|_{L^p(\R^2)}, \qquad 1<p<\infty,
\end{equation}
 with  constant $C$ depending on $p,D$ only.
\end{theorem}
%%%%%%%%%%%%%%%%%%%%%%%%%%%%%% THEOREM THEOREM THEOREM

%%%%%%%%%%%%%%%%%%%%%%%%%%%%%% REMARK REMARK REMARK
\begin{remark}
While we chose a specific $D$-lacunary sequence as the range of ${\mathsf v}_D$ in Theorem \ref{t.Liplac},  namely
\[
{\mathsf v}_D(\R^2) \subset \Bigg\{ \prod_{j=1}^D
\e^{2\pi i 2^{-k_j}}:\; (k_1,\ldots,k_D)\in \mathbb N^D, \; k_1 <  k_2<\ldots<k_D\Bigg\},
\]
the statement and proof of our theorem, as well as that of \cite{GuoThiele}, can be suitably modified to fit a vector field whose range is any given $D$-lacunary set. 
\end{remark}
%%%%%%%%%%%%%%%%%%%%%%%%%%%%%% REMARK REMARK REMARK
\begin{remark} The same Lipschitz-lacunary linearization of \cite{GuoThiele}, namely \eqref{vD} with $D=1$, was employed by Saari and Thiele in the context of $L^p$-bounds for  maximal multipliers with hyperbolic cross-type singularity \cite{ST}. In light of Theorem \ref{t.Liplac} it seems worthwhile   to investigate whether the results of \cite{ST} hold with the more general linearization \eqref{vD}.
\end{remark}

%%%%%%%%%%%%%%%%%%%%%%%%%%%%%% REMARK REMARK REMARK
\begin{remark} In both Theorems  \ref{t.D-lacunary:main} and \ref{t.Liplac}, the inequality for the Hilbert transform along $\mathsf{v}$ is reduced to a  Littlewood-Paley type square function estimate, explicitly in the first case (see Proposition \ref{p.D-lacunary:CWW}, using a version of the Chang-Wilson-Wolff inequality) or implicitly in the second, relying upon the Lipschitz character of $\mathsf{v}$. This reduction is a common theme in the literature \cite{BatThiele,Dem,DDP,LacLi:mem}; on the other hand, the proofs of boundedness for such square functions rely upon certain $L^p$ estimates for   maximal functions oriented along the direction of the vector fields. Our results exhaust the  cases where, due to the finite order lacunarity of the range,  such maximal functions  have a full range of $L^p$-bounds without any further assumption  on the vector field (see \eqref{e.direcmax} below).   \end{remark}
%%%%%%%%%%%%%%%%%%%%%%%%%%%%%% REMARK REMARK REMARK

%%%%%%%%%%%%%%%%%%%%%%%%%% ACKNOWLEDGMENTS^3
\subsection*{Acknowledgments}
This work was completed while first author was in residence at the Basque Center for Applied Mathematics (BCAM), Bilbao as a visiting fellow. The author gratefully acknowledges the kind hospitality of the staff and researchers at BCAM and in particular of Carlos P\'erez. The authors thank Pedro Caro for interesting exchanges on the relation between directional singular integrals and rotational smoothing.
 
%%%%%%%%%%%%%%%%%%%%%%%%%%%%%% SECTION SECTION SECTION
\section{Auxiliary results}
In this section, we collect some  tools which will be employed within the proofs of the main results.   We begin with an important section that concerns the notations and assumptions used in this paper. 

%%%%%%%%%%%%%%%%%%%%%%%%%%%%%% SECTION SECTION SECTION
\subsection{Some conventions concerning lacunary sets of directions}\label{s.conv} Let $\v\subset S^1$ be a $D$-lacunary set of directions. We identify a direction $v\in\v$ with its argument, namely $v\coloneqq e^{2\pi i \theta_v}$. Then $v^\perp=e^{2\pi i (\theta_v+1/4)}\eqqcolon e^{2\pi i \theta_{v^\perp}}$. By a slight abuse of notation we will also write $\theta^\perp\coloneqq \theta+\frac{1}{4}$ so that $\theta_{v^\perp}=\theta_v ^\perp$. Here we remember that a $D$-lacunary set $\v$ is a successor of a $(D-1)$-lacunary set and,  denoting by $'I_k$ the complementary arcs of $'\v$, for each $k$ we have that $\v\cap {'I_k}$ is contained in the union of two lacunary sequences, one clockwise and one counterclockwise. However, by symmetry, it suffices to consider only the clockwise ordered sequences. Thus, all lacunary sets in this paper will be in clockwise order with respect to their limits.

A further convention that we will use in this paper is that the $D$-lacunary sets we consider will lie in the first quadrant so that for each $v\in\v$ we have $0\leq \theta_v\leq 1/4$. Without loss of generality we assume throughout the paper that the root $v_\infty$ of $\v$ is $v_\infty=(1,0)=e^{2\pi i 0}$.

A $1$-lacunary set with lacunarity constant $\lambda$ can be split into roughly $\log (\lambda/\lambda')$ pairwise disjoint $1$-lacunary sets with the same limit and lacunarity constant $\lambda'<\lambda$. Inductively, a $D$-lacunary set with constant $\lambda$ can be split into $O_{\lambda',\lambda,D}(1)$ pairwise disjoint $D$-lacunary sets with constant $\lambda'$. Given a $D$-lacunary set $\v=\cup_\ell \v_\ell$, the maximal Hilbert transform $H_\v$ can be trivially controlled from above by $\sum_\ell H_{\v_\ell}$. As we can tolerate constants depending on $\lambda$ and $D$, we will, without particular mention, assume that the lacunarity constant of a given $D$-lacunary set is sufficiently small, whenever needed. Note that this splitting will only introduce constants depending on $\lambda$ and $D$.

%%%%%%%%%%%%%%%%%%%%%%%%%%%%%% SECTION SECTION SECTION
\subsection{A representation of the Hilbert transform in terms of cone projections}\label{s.cones}
We continue with the introduction of a notation that we will consistently use throughout the paper. Given a $D$-lacunary set $\v=\{v_j\}_{j\geq 1}$ we always write $'\v=\{u_\tau\}_{\tau\geq 1}$ if $\v$ is a successor of $'\v$. Letting again $'I_\tau$ be the complementary arcs of $'\v$, we  have that 
\begin{equation} \label{eq:primeI}
\v=\bigcup_{\tau\geq 1}(\v\cap {'I_\tau})\eqqcolon \bigcup_{\tau\geq 1}\v_\tau
\end{equation}
where each set $\v_\tau$ is a $1$-lacunary set with limit $u_\tau$. We now agree to write 
\[
\v_\tau=\{v_{1,\tau},v_{2,\tau},\ldots,v_{j,\tau},\ldots\}
\]
where $v_{1,\tau} \coloneqq u_{\tau-1}$ for $\tau\geq 2$ and $\lim_{j\to \infty} v_{j,\tau}=u_\tau$. For $\tau=1$ we just use the convention $u_0\coloneqq v_1$ and the same definition makes sense.

Given a $D$-lacunary set $\v=\{v_j\}_j$ we denote by $\{C_j\}_j$ the frequency cones in the third quadrant defined by the complementary arcs of $\v$. More precisely we set
\[
C_j\coloneqq\{\xi\in\R^2:\, \xi\cdot v_j  \geq 0,\, \xi\cdot v_{j+1}<0, \, \xi_1<0,\, \xi_2>0\}.
\]
We often work with $\v$ being a successor of $'\v=\{u_\tau\}_{\tau}$, so that both $\v$ and $'\v$ appear in the course of an argument. In order to avoid ambiguities, we use the redundant notation
\[
\tilde C_\tau \coloneqq\{\xi\in\R^2:\, \xi\cdot u_\tau \geq 0,\, \xi\cdot u_{\tau+1}<0, \, \xi_1<0,\, \xi_2>0\}
\]
for the coarser cones corresponding to the complementary arcs of $'\v$. Finally, as $\v=\cup_{\tau\geq 1} \v_\tau$ as in \eqref{eq:primeI} we will use the doubly-indexed cones $C_{j,\tau}$ to denote those cones among the collection $\{C_j\}_{j}$ that are contained in $\tilde C_\tau$; that is, we set 
\[
C_{j,\tau}\coloneqq\{\xi\in\R^2:\, \xi\cdot v_{j,\tau} \geq 0,\, \xi\cdot v_{j+1,\tau}< 0, \, \xi_1<0,\, \xi_2>0\}.
\]
These should be understood as regions in the frequency plane; see also Figure~\ref{f.lac} below. With these notations in hand we then have the identity
\[
\{\xi\in \R^2: \, \xi_1<0,\, \xi_2>0,\, \xi\cdot v_1\geq 0\} = \bigcup_{\tau\geq 1}\bigcup_{j\geq 1}C_{j,\tau}=\bigcup_{\tau\geq 1} \tilde C_{\tau-1}.
\]

%%%%%%%%%%%%%%%%%%%%%%%%%%%%%% FIGURE FIGURE FIGURE
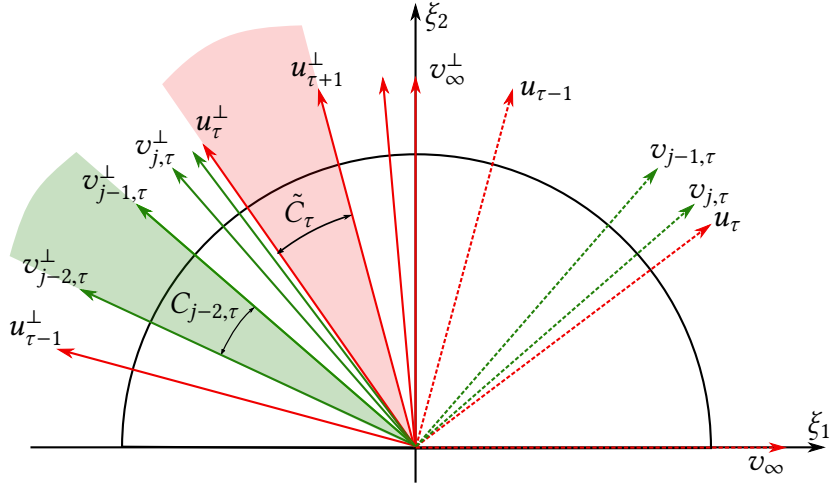
\begin{figure}[htb]
\centering
 \def\svgwidth{330pt}
\begingroup%
  \makeatletter%
  \providecommand\color[2][]{%
    \errmessage{(Inkscape) Color is used for the text in Inkscape, but the package 'color.sty' is not loaded}%
    \renewcommand\color[2][]{}%
  }%
  \providecommand\transparent[1]{%
    \errmessage{(Inkscape) Transparency is used (non-zero) for the text in Inkscape, but the package 'transparent.sty' is not loaded}%
    \renewcommand\transparent[1]{}%
  }%
  \providecommand\rotatebox[2]{#2}%
  \ifx\svgwidth\undefined%
    \setlength{\unitlength}{331.83959961bp}%
    \ifx\svgscale\undefined%
      \relax%
    \else%
      \setlength{\unitlength}{\unitlength * \real{\svgscale}}%
    \fi%
  \else%
    \setlength{\unitlength}{\svgwidth}%
  \fi%
  \global\let\svgwidth\undefined%
  \global\let\svgscale\undefined%
  \makeatother%
  \begin{picture}(1,0.55407278)%
    \put(0,0){\includegraphics[width=\unitlength,page=1]{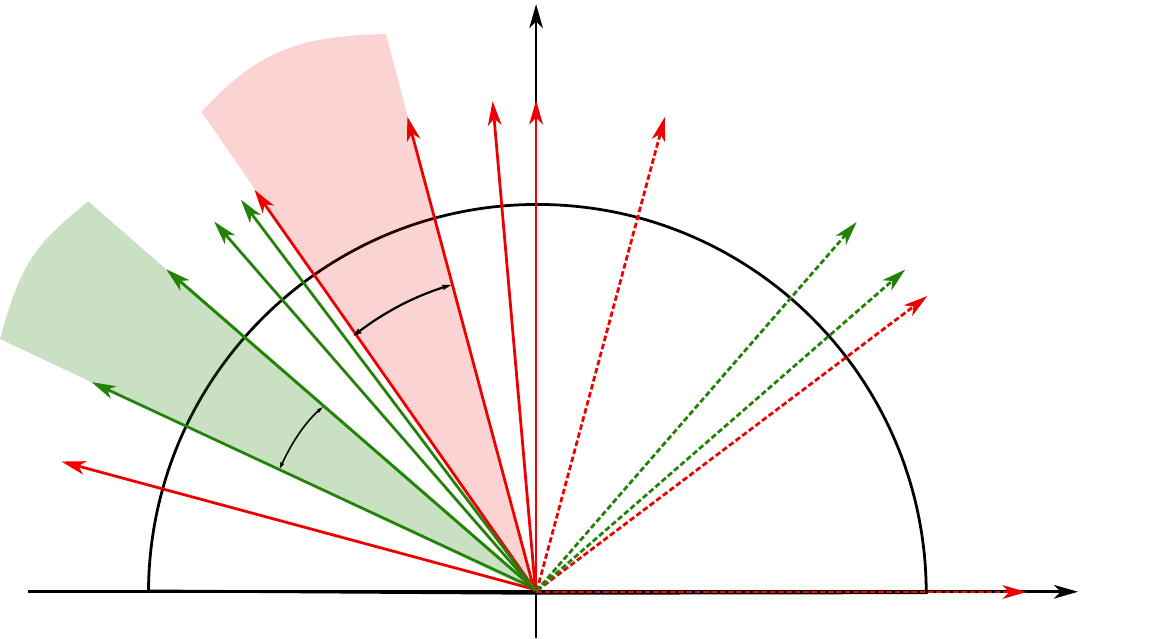}}%
    \put(0.32232973,0.46602321){\color[rgb]{0,0,0}\makebox(0,0)[lb]{\smash{$u_{\tau+1} ^\perp$}}}%
    \put(0.48083011,0.46450285){\color[rgb]{0,0,0}\makebox(0,0)[lb]{\smash{$v_\infty ^\perp$}}}%
    % \put(0,0){\includegraphics[width=\unitlength,page=2]{lac.pdf}}%
    \put(0.31526219,0.29864822){\color[rgb]{0,0,0}\makebox(0,0)[lb]{\smash{$\tilde C_\tau$}}}%
    \put(0.80740207,0.29392689){\color[rgb]{0,0,0}\makebox(0,0)[lb]{\smash{$u_\tau$}}}%
    \put(0.73404465,0.37257942){\color[rgb]{0,0,0}\makebox(0,0)[lb]{\smash{$v_{j-1,\tau}$}}}%
    \put(0.58501944,0.44437262){\color[rgb]{0,0,0}\makebox(0,0)[lb]{\smash{$u_{\tau-1}$}}}%
    \put(0.77841916,0.32634656){\color[rgb]{0,0,0}\makebox(0,0)[lb]{\smash{$v_{j,\tau}$}}}%
    \put(0.21329157,0.40460548){\color[rgb]{0,0,0}\makebox(0,0)[lb]{\smash{$u_{\tau} ^\perp$}}}%
    \put(0.08268519,0.33665979){\color[rgb]{0,0,0}\makebox(0,0)[lb]{\smash{$v_{j-1,\tau} ^\perp$}}}%
    \put(0.14351262,0.37694098){\color[rgb]{0,0,0}\makebox(0,0)[lb]{\smash{$v_{j,\tau} ^\perp$}}}%
    \put(0.47519509,0.5265981){\color[rgb]{0,0,0}\makebox(0,0)[lb]{\smash{$\xi_2$}}}%
    \put(0.9139081,0.05673505){\color[rgb]{0,0,0}\makebox(0,0)[lb]{\smash{$\xi_1$}}}%
    \put(0.84632564,0.01444241){\color[rgb]{0,0,0}\makebox(0,0)[lb]{\smash{$v_\infty$}}}%
    \put(-0.00032423,0.16946439){\color[rgb]{0,0,0}\makebox(0,0)[lb]{\smash{$u_{\tau-1} ^\perp$}}}%
    \put(0.01454192,0.24020618){\color[rgb]{0,0,0}\makebox(0,0)[lb]{\smash{$v_{j-2,\tau} ^\perp$}}}%
    \put(0.1845854,0.19468372){\color[rgb]{0,0,0}\makebox(0,0)[lb]{\smash{$C_{j-2,\tau}$}}}%
  \end{picture}%
\endgroup%
\caption[Lacunary sets and cones]{\small Lacunary sets and cones}\label{f.lac}
\end{figure}
%%%%%%%%%%%%%%%%%%%%%%%%%%%%%% FIGURE FIGURE FIGURE

Given a $D$-lacunary set $\v$ which is successor to $'\v=\{v_j\}_j$ we will write $R_{j}$ for the Fourier restrictions $\widehat{R_{j} f} \coloneqq \widehat f \cic{1}_{C_{j}}$. The Fourier projections $R_{j,\tau}$ and $\tilde R_\tau$ are defined similarly by means of the cones $C_{j,\tau}$, and $\tilde C_{\tau}$, respectively. We have actually proved the following representation for the operator $H_\v ^+$.

%%%%%%%%%%%%%%%%%%%%%%%%%%%%%% LEMMA LEMMA LEMMA
\begin{lemma}\label{l.H^+rep}Let $\v=\{v_j\}_{j\geq 1}$ be a $D$-lacunary set of directions following the assumptions of this paragraph, and assume that $\v$ is successor to $'\v=\{u_\tau\}_{\tau\geq 1}$ and write $\v=\cup_{\tau\geq 1}\v_\tau$ as in \eqref{eq:primeI}.

Then for all $g\in \mathcal S(\R^2)$ such that $\mathrm{supp}\, \widehat g\subseteq \{\xi\in\R^2:\, \xi_1<0,\, \xi_2>0\},$ we have
\[
	H_{\v} ^+ g (x) = \sup_{J,T\geq 1} \Big|\sum_{j\geq J} R_{j,T} g(x)+\sum_{\tau\geq T+1} \tilde R_{\tau} g(x)\Big|.
\]
Furthermore, for $g$ as above we have the recurrence estimate
\[
H_{\v} ^+ g (x)\leq  H^+ _{'\v} (g )+\sup_{\tau\geq 1}H^+ _{\v_\tau}(\tilde R_{\tau-1}g).
\]
\end{lemma}
%%%%%%%%%%%%%%%%%%%%%%%%%%%%%% LEMMA LEMMA LEMMA

%%%%%%%%%%%%%%%%%%%%%%%%%%%%%% PROOF PROOF PROOF
\begin{proof} The first identity of the lemma follows immediately from the the discussion before its statement. To see the second estimate of the lemma, fix some $v\in\v$ and suppose that $v=v_{J,T}$ for some $J,T\geq 1$, where $v_{J,T}$ is defined in the discussion before the statement of the lemma. Using the representation formula of the lemma we easily obtain
\[
\begin{split}
H^+ _{v}(g) &= \sum_{j \geq J} R_{\ell,T} g +\sum_{\tau\geq T+1}\tilde R_\tau g=\sum_{j \geq J}R_{T,j} (\tilde R_{T-1}g)+\sum_{\tau\geq T+1 } \tilde R_\tau(g)
\\
&= \sum_{j\geq J }R_{j,T}(\tilde R_{T-1}g)+H_{u_T} ^+ g.
\\
 \end{split}
\]
Thus we have the estimate
\[
H^+ _\v (g ) \leq \sup_{T\geq 1} |H_{u_T} ^+ g|+ \sup_{\substack{T\geq 1\\J\geq 1 }} \Big|\sum_{j\geq J }R_{j,T}(\tilde R_{T-1}g) \Big|\leq H^+ _{'\v} (g )+\sup_{\tau\geq 1}H^+ _{\v_\tau}(\tilde R_{\tau-1}g)
\]
as desired.
\end{proof}
%%%%%%%%%%%%%%%%%%%%%%%%%%%%%% PROOF PROOF PROOF

%%%%%%%%%%%%%%%%%%%%%%%%%%%%%% SECTION SECTION SECTION
\subsection{Maximal functions and cone frequency projections in $D$-lacunary directions} 
We need to recall two interconnected results, both essentially due to Sj\"ogren and Sjolin \cite[Corollary 2.4]{SS}.  Firstly, as mentioned in the introduction, for each nonnegative integer $D$ and each $1<p\leq \infty$ the directional maximal function $\M_\v$ satisfies the bound 
\begin{equation} \label{e.direcmax}
 \|\M_\v\|_{L^p\to L^p} \leq \tilde \kappa _p(D),
\end{equation}
uniformly over all $D$-lacunary sets of directions $\v$. 

Consider now a $D$-lacunary set of directions $\v=\{v_j\}_j\subset S^1$ and let $\{R_j\}$ be the frequency projections corresponding to the cones $C_j$, as in \S\ref{s.cones}. Then for all $1<p<\infty$ and uniformly over all $D$-lacunary $\v\subset S^1$ we have
\begin{equation} \label{e.SS}
\sup_{\eps_j \in\{-1,0,1\}} \bigg\|\sum_{j} \eps_j R_j f \bigg\|_{L^p(\R^2)}\leq  \kappa_p(D) \|f\|_{L^p(\R^2)}. 
\end{equation}
We stress here that $\kappa_p(D)$ only depends on $p$ and the order of lacunarity $D$. The estimate above is contained within \cite[Corollary 2.4]{SS}; see also\cite[Lemma 3.3]{DDP} for the case $D=1$.
%%%%%%%%%%%%%%%%%%%%%%%%%%%%%% SECTION SECTION SECTION
\subsection{Maximal Fourier multipliers and square functions} The strategy for proving our main results is to first reduce the norm inequality for $H_\v f$ to an appropriate square function estimate that will allow us to independently handle Littlewood-Paley pieces of the function $f$. 

In order to formulate the main square function estimate we first introduce the standard Littlewood-Paley projections. To that end, let $\phi$ be a Schwartz function supported on $[1/2,2]$ and such that
\[
\sum_{k \in \mathbb Z} \phi(2^{-k}\xi ) =1, \qquad \forall \xi \in \R\backslash\{0\}.
\]
For a function $f\in C^\infty_0(\R^2)$ set $\widehat{S_k f}(\xi)\coloneqq\widehat f(\xi ) \phi(2^{-k}|\xi |).$  The following proposition is a  consequence of the Chang-Wilson-Wolff inequality \cite{CWW}. The proof essentially follows from \cite[Lemma 3.1]{GHS}; see also \cite[Proposition 3.1]{DDP}. In the formulation that follows each operator $P_j$ is a Fourier multiplier operator, $\widehat{P_j f}(\xi)\coloneqq m_j(\xi)\widehat f(\xi)$.
%%%%%%%%%%%%%%%%%%%%%%%%%%%%%% PROPOSITION PROPOSITION PROPOSITION
\begin{proposition}\label{p.D-lacunary:CWW}
Let $\{P_j\}_{j=1} ^N$ be Fourier multiplier operators which for all $j$ satisfy  $|\partial_\xi ^\alpha m_j(\xi)|\leq C_\alpha |\xi|^{-|\alpha|}$ for all multi-indices $\alpha$, with $C_\alpha$ independent of $j$. Then for all $1<p<\infty$ there exists $C_p>0$, depending only on $p$, such that
\[
\Big\|\sup_{j=1,\ldots,N} |P_j f|\Big\|_{L^p(\R^2)} \leq C_p (\log N)^{\frac12} \bigg\|  \bigg( \sum_{k \in \mathbb Z}  \Big(\sup_{j=1,\ldots,N} |P_j(S_k f)|\Big)^2 \bigg)^{\frac{1}{2}} \bigg\|_{L^p(\R^2)}.
\]
\end{proposition}
%%%%%%%%%%%%%%%%%%%%%%%%%%%%%% PROPOSITION PROPOSITION PROPOSITION

%%%%%%%%%%%%%%%%%%%%%%%%%%%%%% SECTION SECTION SECTION
\subsection{A single annulus pointwise estimate}
The next lemma  relates the maximal operator associated with a 1-lacunary set, when localized to a frequency annulus,  to the maximal function in the limiting direction. The proof is a simpler version of the argument used for the proof of  \cite[Lemma 3.2]{DDP}.

%%%%%%%%%%%%%%%%%%%%%%%%%%%%%% LEMMA LEMMA LEMMA
\begin{lemma} \label{l.lemmamf}Assume that $ \{v_{j}\}_{j\geq 1}$ is a $1$-lacunary set with limit $v$, following the conventions of \S\ref{s.conv}. Let $f\in\mathcal S(\R^2)$ with $\mathrm{supp}\,\widehat f \subseteq \{\xi\in \R^2:\, \xi\cdot v_1\geq 0,\, \xi\cdot v_\infty < 0\}$. With $H_v ^+$ defined as in \eqref{e.Hilb+} we have for every $k\in\mathbb Z$
\[
\sup_{j} |H^+ _{v_j}(S_k f)(x)| \lesssim \M_{v } [(S_k f)_{\mathrm{odd}}](x) +   \M_{v }[(S_k f)_{\mathrm{ev}}](x)
\]
where $ g_{ \mathrm{odd}}\coloneqq \sum_{j \, \mathrm{odd}} R_j g$ and similarly for $g_{ \mathrm{ev}}$.
\end{lemma}
%%%%%%%%%%%%%%%%%%%%%%%%%%%%%% SECTION SECTION SECTION
\subsection{A Fefferman-Stein inequality for directional maximal functions} Let $\v=\{v_j\}_{j\geq 1}$ be a $D$-lacunary set of directions for some nonnegative integer $D$. The following Fefferman-Stein type inequality will play an important role in deducing the square function estimate, which is the main estimate for the proof of Theorem~\ref{t.D-lacunary:main}. We include a proof for completeness although we do not claim any originality on this result.

%%%%%%%%%%%%%%%%%%%%%%%%%%%%%% LEMMA LEMMA LEMMA
\begin{lemma}\label{l.FS} Let $\v=\{v_j\}_{j\geq 1}$ be a $D$-lacunary set of directions. For all $1<p,q<\infty$ we have
\[
\| \{\M_{v_j  }  h_j \}  \|_{L^p(\R^2;\, \ell^q )} \leq K_{q,p}(D)     \| \{  h_j \}  \|_{L^p(\R^2;\, \ell^q )} 
\]
where the constant $K_{q,p}(D)$ depends only on $q,p,D$.
\end{lemma}
%%%%%%%%%%%%%%%%%%%%%%%%%%%%%% LEMMA LEMMA LEMMA

%%%%%%%%%%%%%%%%%%%%%%%%%%%%%% PROOF PROOF PROOF
\begin{proof} It suffices to prove the same bound for the linearized operators
\[
f\mapsto   \frac{1}{2\eps(x)} \int_{-\eps(x)}^{\eps(x)} |f(x+tu_j)| \,\d t
\]
which we continue to denote by $\M_{v_j}$, uniformly over all positive measurable functions $\eps(x)$. This formulation has the advantage of allowing for complex interpolation. The case $p\leq q$ of the estimate of the lemma follows by complex interpolation of the following estimates for the vector valued operator $\{h_j\}\mapsto \{M_{v_j  } h_j\}$; namely the trivial estimate
 \[
\| \{\M_{v_j  } h_j \} \|_{L^p(\R^2;\, \ell^p)} \leq \|\M_{\v}\|_{L^p\to L^p} \| \{  h_j \} \|_{L^p(\R^2;\, \ell^p)} \lesssim_{p,D}  \| \{  h_j \} \|_{L^p(\R^2;\, \ell^p)},
\]
with the easy estimate
\[
\| \{\M_{v_j} h_j\}  \|_{L^p(\R^2;\, \ell^\infty )} \leq  \| \M_{\v} (\sup_j |h_j|) \|_{L^p(\R^2)}  \lesssim_{p,D} \| \{h_j\} \|_{L^p(\R^2;\;  \ell^\infty )} .
\] 
Observe that for $\theta=1-p/q  $ we have $\theta\in[0,1)$ and $[L^p(\R^2; \ell^p):L^p(\R^2; \ell^\infty)]_\theta=L^p(\R^2; \ell^q)$.

We now prove the case $p> q$. Let $r=(p/q)'> 1$. We start by observing that
\[
\| \{\M_{v_j  }  h_j \}  \|_{L^p(\R^2;\, \ell^q )}  ^q=\bigg( \int_{\R^2} \Big(\sum_{j} |\M_{v_j}h_j(x)|^q\Big)^\frac{p}{q}\,\d x\bigg)^ \frac{q}{p}=\int_{\R^2} \sum_{j} |\M_{v_j}h_j(x)|^q u(x)\,\d x
\]
for some nonnegative $u \in L^r(\R^2)$ with $\|u \|_r=1$. Given a Schwartz function $f$ and a nonnegative function $w$, the Fefferman-Stein inequality
\[
\int_{\R^2} |\M_{v_j} f|^q w \lesssim_q  \int_{\R^2}|f|^q \M_{v_j}w
\]
follows easily by its one-dimensional counterpart along the direction $v_j\in S^1$. Note that the estimates produced that way are independent of $v_j \in S^1$. Thus we can estimate
\[
\begin{split}
\int_{\R^2}\Big( \sum_{j} |\M_{v_j}   h_j(x)|^q\Big) u(x) \, \d x &\lesssim_q \sum_{j} \int_{\R^2} |h_j(x)|^q \M_{v_j} u(x)\, \d x
\leq \int_{\R^2} \sum_j |h_j(x)|^q \M_{\v} u(x)\, \d x
\\
& \leq \Big(\int_{\R^2} \big(\sum_j |h_j(x)|^q\big)^\frac{p}{q}\Big)^\frac{q}{p} \| \M_{\v} u\|_{L^r(\R^2)}
\\
&\leq   \|\{h_j\}\|_{L^p(\R^2;\,\ell ^q)}^q\| \M_{\v} \|_{L^r \to L^r}\|u\|_{L^r(\R^2)}\lesssim_{p,q,D}  \|\{h_j\}\|_{L^p(\R^2;\,\ell ^q)}^q
\end{split}
\]
which gives the desired Fefferman-Stein inequality for $p>q$ by taking $q$-th roots on both sides of the estimate. 
\end{proof}
%%%%%%%%%%%%%%%%%%%%%%%%%%%%%% PROOF PROOF PROOF

%%%%%%%%%%%%%%%%%%%%%%%%%%%%%% SECTION SECTION SECTION
\section{Proof of Theorem~\ref{t.D-lacunary:main}} This section is devoted to the proof of Theorem~\ref{t.D-lacunary:main}. 
Let us immediately observe that the upper bound of the theorem together with Karagulyan's $L^2\to L^2$ lower bound from \cite{Karag} and Riesz-Thorin interpolation immediately implies the lower bound for all $p\in(1,\infty)$. Thus it suffices to prove the upper bound in Theorem \ref{t.D-lacunary:main} for all $1<p<\infty$.  We will actually prove the theorem for the operator $H_\v ^+$ as defined in~\eqref{e.Hilb+} which is a linear combination of the identity and $H_\v$, and thus obeys equivalent norm bounds.

%%%%%%%%%%%%%%%%%%%%%%%%%%%%%% SECTION SECTION SECTION
\subsection{Initial reductions}
Throughout, we fix a $D$-lacunary set of directions $\v=\{v_j\}_{j\geq 1}$ with sufficiently small lacunarity constant which is a successor of a $(D-1)$-lacunary set $'\v=\{u_\tau\}_{\tau\geq 1},$ so that $\v=\cup_{\tau\geq 1}\v_\tau$, according to the conventions and notations of \S\ref{s.conv}. Here we remember that for each $\tau\geq 1$ we have $\v_\tau\coloneqq \v\cap {'I_\tau}=\{v_{j,\tau}\}_{j\geq 1}$, where $\{'I_\tau\}_{\tau\geq 1}$ are the complementary arcs of $'\v$. With these assumptions taken as understood we may reduce to the case where $\widehat f$ is supported on the second and fourth quadrant. Indeed, since we assume that the directions of $\v$ are in the first quadrant, the operators $H_{v_j} ^+$ do not touch functions with frequency support in the first and third quadrants. By symmetry, we can further reduce to the case that $ f$ has frequency support in the second quadrant. 

Applying Proposition~\ref{p.D-lacunary:CWW} to the family $\{H^+_{v}\}_{v\in\v}$, the upper bound in Theorem~\ref{t.D-lacunary:main} reduces to the square function estimate
\begin{equation} \label{D-lacunary:SFE}
 \bigg\|  \Big( \sum_{k \in \mathbb Z} |H_{\v} ^+ (S_k f)|^2 \Big)^{\frac{1}{2}} \bigg\|_{L^p(\R^2)} \leq C_p(D)\|f\|_{L^p(\R^2)}, \qquad1< p < \infty,
\end{equation} 
for any countable set of directions $\v$ which is $D$-lacunary, where the constant $C_p(D)$ depends only on $p,D$. Here we remember that $S_k$ are smooth frequency projections to the annulus $\{|\xi|\simeq 2^k\}$.

%%%%%%%%%%%%%%%%%%%%%%%%%%%%%% SECTION SECTION SECTION
\subsection{Proof of \eqref{D-lacunary:SFE}} The main part of the proof proceeds by induction of \eqref{D-lacunary:SFE} on the order of lacunarity. For $D=1$ the base step of the induction follows from \cite[\S3.2]{DDP}. Now let $D\geq 2$ and assume that \eqref{D-lacunary:SFE} holds uniformly for all $(D-1)$-lacunary sets of directions. The recurrence estimate of Lemma~\ref{l.H^+rep} applied to $g= S_k f$ gives
\[
H_{\v} ^+ S_k f (x)\leq H_{'\v} ^+ S_k f + \sup_{\tau\geq 1} H_{\v_\tau} ^+ (\tilde R_{\tau-1}S_k f).
\]
This together with the inductive hypothesis now imply that
\begin{equation}\label{e.square}
\begin{split}
\Big\|\Big( \sum_{k \in \mathbb Z} |H_{\v} (S_k f)|^2 \Big)^{\frac12} \Big\|_{L^p(\R^2)}  &\lesssim  C_p(D-1 )\|f\|_{L^p(\R^2)}   
\\
&\qquad + \bigg\|\Big( \sum_{k \in \mathbb Z}  \sup_{\tau\geq 1}  | H_{\v_\tau} ^+(\tilde R_{\tau-1} S_k f)|^2 \Big)^{\frac12} \bigg\|_{L^p(\R^2)}.
\end{split}
\end{equation}

We now proceed with the estimate of the second summand in the right hand side above. According to Lemma \ref{l.lemmamf}, we have for each $\tau\geq 1$ that
\begin{equation}\label{Mus}
\begin{split}
H^+ _{\v_\tau} (\tilde R_{\tau-1}  S_k f)  & \lesssim \M_{u_\tau } [(\tilde R_{\tau-1}S_k f)_{ \mathrm{odd}}] +   \M_{u_\tau} [(\tilde R_{\tau-1}S_k f)_{ \mathrm{ev}}] 
\\
&\lesssim  \Big(\sum_{\tau\geq 1}  \big|\M_{u_\tau}(F_{\mathrm{1}} ^{k,\tau})\big|^2\Big)^\frac{1}{2} + \Big(\sum_{\tau\geq 1}  \big|\M_{u_\tau}(F_{\mathrm{2}} ^{k,\tau})\big|^2\Big)^\frac{1}{2},
\end{split}
\end{equation}
where $F_\mathrm{1} ^{k,\tau}\coloneqq (\tilde R_{\tau-1}S_k f)_{ \mathrm{odd}} $ and $F_\mathrm{2} ^{k,\tau}\coloneqq (\tilde R_{\tau-1}S_k f)_{ \mathrm{even}} $. We only treat the first summand in the estimate above, the estimate for the second one being similar. Observe that
\[
F_\mathrm{1} ^{k,\tau}= \sum_{j \, \mathrm{odd}} R_{j,\tau} (\tilde R_{\tau-1}S_k f).
\]
Applying the case $q=2$ of the Fefferman-Stein inequality of Lemma~\ref{l.FS} to the $(D-1)$-lacunary sequence $\{u_\tau\}_{\tau\geq 1}$ yields
\[
\bigg\|\bigg(\sum_{k\in \mathbb Z}\sum_{\tau\geq 1} |\M_{u_\tau  } (F_{\mathrm{1}} ^{k,\tau})|^2\bigg)^{\frac{1}{2}}\bigg\|_{L^p(\R^2)} \leq K_p(D-1)   \bigg\|\bigg(\sum_{k} \sum_{\tau} |F_{\mathrm{1}} ^{k,\tau} |^2\bigg)^{\frac{1}{2}}\bigg\|_{L^p(\R^2)}.
\]
However, the right hand side of the previous estimate can be estimated by the biparameter Khintchine inequality as follows
\[
\begin{split}
	\bigg\|\bigg(\sum_{k} \sum_{\tau} |F_{\mathrm{1}} ^{k,\tau} |^2\bigg)^{\frac{1}{2}}\bigg\|_{L^p(\R^2)} ^p & \lesssim \Exp_\eta \Exp_\omega \int_{\R^2} \Bigg| \sum_{\tau\geq 1} \sum_{{j \, \mathrm{odd}} }\eps_\tau(\eta)   R_{j,\tau} \Big( \sum_{k\in\mathbb Z} \eps_k(\omega)S_k f \Big) \Bigg|^p
	\\
	&\leq \kappa_p(D-1) ^p \Exp_\omega \int_{\R^2}  \Big| \sum_{k\in\mathbb Z} \eps_k(\omega)S_k f \Big|^p\simeq_{p} \kappa_p( D-1)^p \|f\|_{L^p(\R^2)} ^p,
\end{split}
\]
where we used \eqref{e.SS} in going from the first to the second line and the Littlewood-Paley inequality in the last estimate. Arguing similarly for $F_{\mathrm 2} ^{k,\tau}$ we have that
\[
\sum_{\mathrm{par}\in\{1,2\}}\bigg\|\bigg(\sum_{k\in \mathbb Z}\sum_{\tau\geq 1} |\M_{u_\tau  } ( F_{\mathrm{par}} ^{k,\tau})|^2\bigg)^{\frac{1}{2}}\bigg\|_{L^p(\R^2)} \lesssim_{D,p} \|f\|_{L^p(\R^2)}
\]
Using these estimates in \eqref{Mus} and replacing in \eqref{e.square} yields \eqref{D-lacunary:SFE} and thus completes the proof of Theorem~\ref{t.D-lacunary:main}.

%%%%%%%%%%%%%%%%%%%%%%%%%%%%%% SECTION SECTION SECTION
\section{Proof of Theorem~\ref{t.Liplac}} This section is devoted to the proof of Theorem~\ref{t.Liplac}. We begin with some reductions of the problem which follow closely the arguments in \cite[\S3]{GuoThiele}.

%%%%%%%%%%%%%%%%%%%%%%%%%%%%%% SECTION SECTION SECTION
\subsection{Preliminary reductions} We begin by setting, for $d=1,\ldots,D$,
\[
\mathsf{v}_d(x)\coloneqq  \prod_{j=1}^d \e^{2\pi i 2^{\lfloor \log_2 \lambda_j(x)\rfloor}},\qquad \v_{d}\coloneqq \mathsf{v}_{d}(\R^2),
\]
so $\v_{d}$ is the range of $\mathsf{v}_{d}$ for each $d$. Notice that, by splitting $f$ into a finite number of disjointly supported pieces, and rotating each piece by a fixed angle, we can assume that $\lambda_1:\R^2\to(0,2^{-15}]$. Then $\lambda_j\leq 2^{-5(j-1)}2^{-15}$ for each $j\geq 2$ and each $\v_{d}$, $d=1,\ldots,D$, is contained in the first quadrant portion of the cone $C_{\mathrm{nh}}$ bordered by the directions $\{\e^{2\pi i 0}, \e^{2\pi i 2^{-5}}\}$. We later refer to sets of directions $\v$ contained in $C_{\mathrm{nh}}$ as being \emph{nearly horizontal}.

Consider the cone in the second   quadrant
 \[
\Gamma_0\coloneqq \{\xi \in \R^2:\, \xi_2>-2\xi_1,\, \xi_1<-2^6,\, \xi_2>0\}.
 \]
The arguments of \cite[\S 3]{GuoThiele}, which apply to any vector field with nearly horizontal range, yield the pointwise bound
\[
|H_{\mathsf{v}_D ,1} f | \lesssim  \M_{(0,1)}\M_{(1,0)} f + \M_{\v_D} f 
\]  
whenever $  \widehat f $ is supported off $\Gamma_0\cup -\Gamma_0$. As the right hand side of the last display is $L^p$-bounded, and by treating similarly the part supported in the lower half-plane,   we are allowed to assume that $\mathrm{supp}\,\widehat f\subset \Gamma_0$ for the remainder of the proof.  In this case, following again \cite{GuoThiele}, we have 
\[
\|H_{\mathsf{v}_D,1} f \|_p \lesssim  \sup_{|\tau|\leq 2^{4} } \|H_{\mathsf{v}_D}^{+\tau} f \|_p   + \| \M_{\v_D} f\|_p
\]
where we are using the notation, for a generic  direction $v$  and vector field $\mathsf v$,
\[
H_{v}^{+\tau} f(x) =\int_{\R^2} \widehat f(\xi) \cic{1}_{[\tau,\infty)} (\xi \cdot v) \, \e^{2\pi i x\cdot \xi }\d \xi, \qquad 
H_{\mathsf v}^{+\tau} f(x)\coloneqq H_{\mathsf v(x)}^{+\tau} f(x).
\]
We have thus reduced Theorem 2 to the proof of the estimate
\begin{equation}
\label{e.T2:main1}
\|H_{\mathsf{v}_D}^{+\tau} f\|_{p} \leq \beta_p(D) \|  f\|_{p}
\end{equation}
where the constant $\beta_p(D)$ can be chosen uniform over $|\tau|\le 2^4$, whenever $\mathrm{supp}\,\widehat f\subset \Gamma_0$. In order to simplify the notation we will we work in the paradigmatic case $\tau=0$, so we only show \eqref{e.T2:main1} for $H^+$. However, the proof for $H^{+\tau}$ is essentially identical and the details are left to the reader.

%%%%%%%%%%%%%%%%%%%%%%%%%%%%%% REMARK REMARK REMARK
\begin{remark} In what follows we will consider the Hilbert transform  $H_{v} f$ along almost horizontal directions $v$. Now it is not hard to verify that $H_v f(x)$  does not depend on $v$ when $f$ has frequency support in $\{ \xi_2 \leq -2\xi_1 \}$.  Thus, instead of the cone $\Gamma_0$ we could have considered the cone $\Gamma_1\coloneqq \{\xi_1<-2^6,\,\xi_2>0\}$. 
	
A consequence of the definition of $\Gamma_0$ is that for every $\xi\in\Gamma_0$ and every almost horizontal direction $v$ we have that
\begin{equation}\label{almostradial}
 |\xi|\leq 2\xi \cdot v^\perp.
\end{equation}
\end{remark}
%%%%%%%%%%%%%%%%%%%%%%%%%%%%%% REMARK REMARK REMARK

%%%%%%%%%%%%%%%%%%%%%%%%%%%%%% SECTION SECTION SECTION
\subsection{Reduction of \eqref{e.T2:main1} to a lacunary vector-valued estimate}
In the next paragraph, we will prove the following proposition.
%%%%%%%%%%%%%%%%%%%%%%%%%%%%%% PROPOSITION PROPOSITION PROPOSITION
\begin{proposition}\label{p.tvv} Let $\v=\{v_\theta\}_{\theta\in\v}=\{e^{2\pi i \theta} \}_{\theta \in \v}$ be a finite almost horizontal $D$-lacunary set of directions. For each $\theta \in \Theta$, let $\lambda_\theta:\R^2\to (0,2^{-15}]$ be  Lipschitz functions with $\|u_\theta\|_{\mathsf{LIP}} \leq 1$ and  define
\[
\mathsf{v}_\theta:\R^2 \to S^1, \qquad \mathsf{v}_\theta(x) \coloneqq \e^{2\pi i \theta} \e^{2\pi i 2^{\lfloor \log_2 \lambda_\theta(x)\rfloor}}.
\]
For every $p\in(1,\infty)$ there exists a constant $B_p(D)$, depending only on $p$ and $D$, such that
\[
\left\|\{ H_{\mathsf v_\theta}^{} f_\theta\} \right\|_{L^p(\R^2;\,\ell^2(\Theta))} \leq B_p(D) \left\|\{   f_\theta\} \right\|_{L^p(\R^2;\,\ell^2(\Theta))}, \qquad 1<p<\infty,
\]
whenever the vector $\{f_\theta \}_{\theta \in \Theta}$ has frequency support in $\Gamma_0$. 
\end{proposition}
%%%%%%%%%%%%%%%%%%%%%%%%%%%%%% PROPOSITION PROPOSITION PROPOSITION

Before the proof of  Proposition \ref{p.tvv}, let us see how the proposition is used to establish \eqref{e.T2:main1}. First observe that the assumption $\lambda_{j-1}(x)\leq 2^{-5}\lambda_j(x)$ and implies that $\v_D=\mathsf{v}_D(\R^2)$ is a $D$-lacunary set which is a successor of $\v_{D-1}=\mathsf{v}_{D-1}(\R^2)$. Let us write $\v_D=\{e^{2\pi i \theta_j}\}_{j\geq 1}$ and $\v_{D-1}=\{e^{2\pi i \gamma_\tau}\}_{\tau \geq 1}$ so that $\v_D$ is a successor to $\v_{D-1}={'\v_D}$. Using the notations of \S\ref{s.conv} we can then write $\v_D=\cup_{\tau\geq 1} \v_{D,\tau}$ where each $\v_{D,\tau}$ is a $1$-lacunary set with limit $e^{2\pi i \gamma_\tau}$ and each $\v_{D,\tau}$ lives in a different complementary arc of $\v_{D-1}$. It is essential to note at this point that, if $v_j\coloneqq e^{2\pi i \theta_j}\in \v_{D,\tau}$ for some (unique) $\tau$ and $e^{2\pi i \theta_j} =\mathsf {v} _D(x)$ then $u_\tau \coloneqq e^{2\pi i \gamma_\tau}=\mathsf{v}_{D-1}(x)$ and
\[
 e^{2\pi i \theta _j} = e^{2\pi i \gamma_\tau } e^{2\pi i 2^{\lfloor \log_2 \lambda_D(x) \rfloor}}\eqqcolon e^{2\pi i \gamma_\tau } \alpha(x) .
\]
Now let $v_j\in \v_D$ so that $v_j\in \v_{D,\tau}$ for a unique $\tau\geq 1$ and let $x$ be such that $\mathsf{v}_D(x)=v_j$ and $\mathsf{v}_{D-1}(x)=u_\tau$. By the previous considerations and an inspection of the proof of Lemma~\ref{l.H^+rep} we get for such $x$ fixed
\[
\begin{split}
	|H^+ _{\mathsf v_{D}(x)}f(x)|&=|H^+ _{v_j}f(x)|\leq |H^+ _{u_\tau}f(x)|+|H^+ _{v_j }\tilde R_{\tau-1}f(x)|
	\\
	& =  |H^+ _{\mathsf{v}_{D-1}(x)} f(x)| + |H^+ _{u_\tau \alpha(x) }\tilde R_{\tau-1}f(x)|
	\\
	& \leq |H^+ _{\mathsf{v}_{D-1}(x)} f(x)| +\Big(\sum_{\tau\geq 1} |H^+ _{u_\tau \alpha(x) }\tilde R_{\tau-1}f(x)|^2\Big)^\frac{1}{2}.
\end{split}
\]
We now use the above relation to find inductively $\beta_p(D)$ so that  \eqref{e.T2:main1} holds. If $D=0$, then $H^+ _{\mathsf{v}_D}=H^+ _{\e^{2\pi i 0}}$ is the usual Hilbert transform on $\R$, therefore \eqref{e.T2:main1} holds with $\beta_p(0)\simeq\max(p,p/(p-1))$. Suppose now that $\beta_p(D-1)<\infty$. Using the previous estimate together with Proposition~\ref{p.tvv} applied to the $(D-1)$-lacunary set $\v_{D-1}=\{e^{2\pi i \gamma_\tau}\}_{\tau\geq 1 }$ and the vector field $\mathsf{v_\tau}(x)\coloneqq e^{2\pi i \gamma_\tau}\alpha(x)$, we get
\begin{equation}
\label{T2:square}
\begin{split}
\|H_{\mathsf{v}_D} ^+ f\|_p &  \leq  \|H_{\mathsf{v}_{D-1}} f\|_{L^p(\R^2)} +   \|\{H_{u_\tau \alpha(x)} (\tilde R_{\tau-1}f)\} \|_{L^p(\R^2;\,\ell^2_\tau)} 
\\ 
&\leq \beta_p(D-1)\|f\|_{L^p(\R^2)} + B_p(D-1)  \|\{  \tilde R_{\tau-1}f\} \|_{L^p(\R^2;\,\ell^2_\tau)}
\\ 
&\leq [\beta_p(D-1)+ B_p(D-1) \kappa_p(D-1) ] \| f \|_{L^p(\R^2)},
\end{split}
\end{equation}
where we applied the Sj\"ogren-Sjolin cone multiplier estimate~\eqref{e.SS} in the last inequality. So we may take $\beta_p(D)$ to be the  bracketed expression in the last display. This completes the induction in the proof of \eqref{e.T2:main1} and therefore the proof of Theorem~\ref{t.Liplac}, up to showing the vector valued estimate of Proposition \ref{p.tvv}.

%%%%%%%%%%%%%%%%%%%%%%%%%%%%%% SECTION SECTION SECTION
\subsection{Proof of  Proposition \ref{p.tvv}} All the implicit constants appearing below are allowed to depend on the lacunarity order $D$ and on exponents $p$ without explicit mention. 

%%%%%%%%%%%%%%%%%%%%%%%%%%%%%% SECTION SECTION SECTION
\subsubsection{Littlewood-Paley projections and preliminary reductions}Throughout the proof, we will adopt the following notation. Let $\phi,\psi$ be  even real Schwartz functions on $\R$ satisfying the following assumptions:  $\widehat \phi$ is supported on $\{1\leq |t|\leq 2\}$ and normalized so that
\[
\sum_{k} \widehat\phi(2^{-k}t) =1, \qquad t\in\R\setminus\{ 0\},
\]
while $\psi$ is spatially supported on $[-1,1]$ and normalized so that
\[
\int_{\R} \psi =0, \qquad \sum_{k} |\widehat\psi(2^{-k}t)|^2 =1, \qquad  t\in\R\setminus\{ 0\}.
\]
Denote by $\phi_k(\cdot)=2^k\phi(2^k\cdot),\psi_k(\cdot)\coloneqq 2^k\psi(2^k\cdot) $. For a direction $\theta \in S^1$, we then introduce the corresponding Littlewood-Paley projections in the direction $\theta$ 
\[\begin{split}
&\Phi_{\theta,k} f(x) \coloneqq \int_{\R} f(x-t \e^{2\pi i \theta }) \phi_{k}(t) \, \d t, 
\\
&\Psi_{\theta ,k} f(x) \coloneqq \int_{\R} f(x-t \e^{2\pi i \theta}) \psi_{k}(t) \, \d t, \qquad \widetilde{\Psi}_{\theta,k}\coloneqq \Psi_{\theta,k} \circ \Psi_{\theta,k}.
\end{split}
\]
We will also use the cutoffs
\[
A_{\theta,k} f(x) \coloneqq\sum_{\tau \leq k}\int_{\R} f(x-t \e^{2\pi i \theta }) \phi_{\tau}(t) \, \d t, \qquad B_{\theta,k}\coloneqq  \mathrm{Id}- A_{\theta,k} .
\]
Note that the frequency support of $A_{\theta,k} $ is contained in a band of width $2^{k+2}$ around the line through the origin oriented along $\theta$, and which acts as the identity on functions with frequency support contained in a band of width $2^{k}$ and same orientation. 
Finally we denote for $j\in\mathbb N$ and $x\in\R^2$
\[
v_{\theta,j}\coloneqq \e^{2\pi i \theta} \e^{2\pi i 2^{-j}}, \qquad  j_\theta(x)\coloneqq-\lfloor \log_2 \lambda_\theta(x) \rfloor, \qquad 
E_{\theta,j}\coloneqq  \{x\in \R^2: \, j_\theta(x) = j\}.
\] 
The assumptions on $\phi,\psi$ together with the hypothesis $\mathrm{supp}\, \widehat{f_\theta}\subset \Gamma_0$ guarantee that for all $\theta,j$
\[
f_\theta = \sum_{k\geq 5} \sum_{\ell\in \mathbb Z} \Phi_{\theta,k}\widetilde{\Psi}_{\theta^\perp,k+j+\ell} f_\theta.
\]
 
%%%%%%%%%%%%%%%%%%%%%%%%%%%%%% FIGURE FIGURE FIGURE
\begin{figure}[htb]
\centering
\def\svgwidth{290pt}
\begingroup%
  \makeatletter%
  \providecommand\color[2][]{%
    \errmessage{(Inkscape) Color is used for the text in Inkscape, but the package 'color.sty' is not loaded}%
    \renewcommand\color[2][]{}%
  }%
  \providecommand\transparent[1]{%
    \errmessage{(Inkscape) Transparency is used (non-zero) for the text in Inkscape, but the package 'transparent.sty' is not loaded}%
    \renewcommand\transparent[1]{}%
  }%
  \providecommand\rotatebox[2]{#2}%
  \ifx\svgwidth\undefined%
    \setlength{\unitlength}{335.65005378bp}%
    \ifx\svgscale\undefined%
      \relax%
    \else%
      \setlength{\unitlength}{\unitlength * \real{\svgscale}}%
    \fi%
  \else%
    \setlength{\unitlength}{\svgwidth}%
  \fi%
  \global\let\svgwidth\undefined%
  \global\let\svgscale\undefined%
  \makeatother%
  \begin{picture}(1,0.64144162)%
    \put(0.74864997,0.28127738){\color[rgb]{0,0,0}\makebox(0,0)[lt]{\begin{minipage}{0.42901825\unitlength}\raggedright \end{minipage}}}%
    \put(0.84398736,0.68646128){\color[rgb]{0,0,0}\makebox(0,0)[lb]{\smash{}}}%
    \put(0.57661593,0.98497401){\color[rgb]{0,0,0}\makebox(0,0)[lt]{\begin{minipage}{0.16684036\unitlength}\raggedright \end{minipage}}}%
    \put(0,0){\includegraphics[width=\unitlength,page=1]{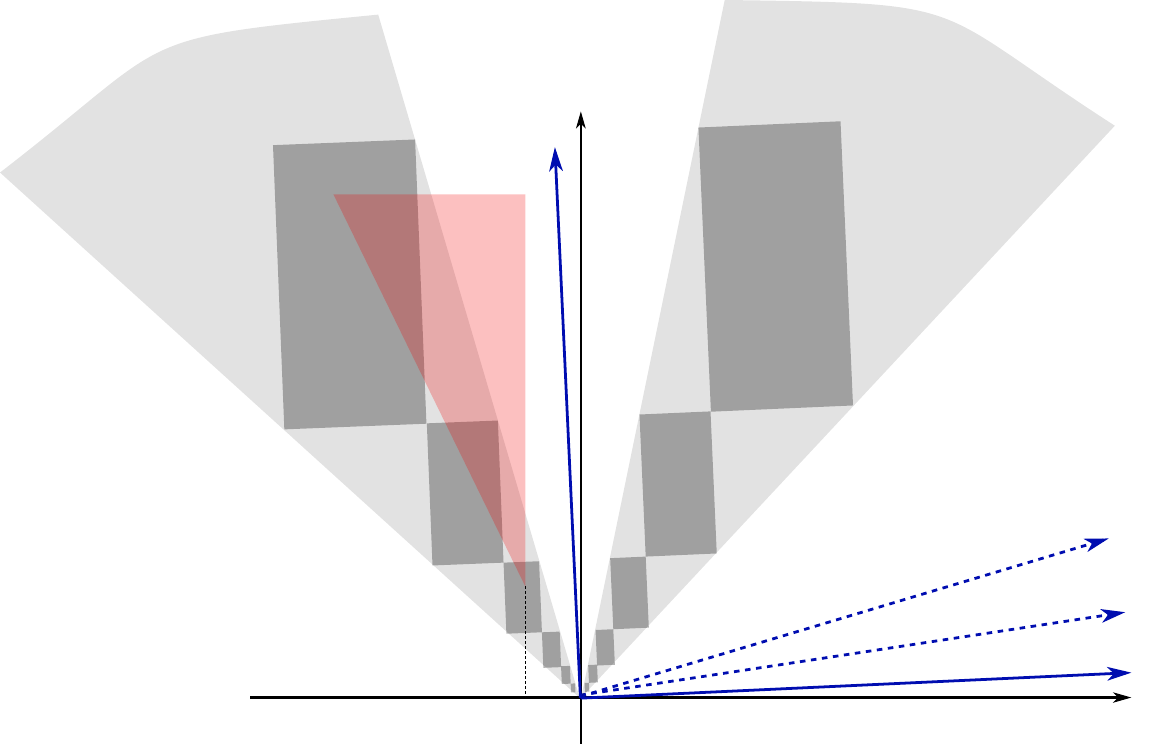}}%
    \put(0.51772062,0.52352506){\color[rgb]{0,0,0}\makebox(0,0)[lb]{\smash{$\xi_2$}}}%
    \put(0.41642232,0.00218937){\color[rgb]{0,0,0}\makebox(0,0)[lb]{\smash{$-2^6$}}}%
    \put(0.38322509,0.37112424){\color[rgb]{0,0,0}\makebox(0,0)[lb]{\smash{$\Gamma_0$}}}%
    \put(0.41801212,0.52317011){\color[rgb]{0,0,0}\makebox(0,0)[lb]{\smash{$v_\theta ^\perp$}}}%
    \put(0.64310705,0.56509254){\color[rgb]{0,0,0}\makebox(0,0)[lb]{\smash{$C_{\theta,j+\ell}$}}}%
    \put(0.94597267,0.00416752){\color[rgb]{0,0,0}\makebox(0,0)[lb]{\smash{$\xi_1$}}}%
    \put(0.97488781,0.06905327){\color[rgb]{0,0,0}\makebox(0,0)[lb]{\smash{$v_\theta$}}}%
    \put(0.95302798,0.19521427){\color[rgb]{0,0,0}\makebox(0,0)[lb]{\smash{$v_{\theta,j}$}}}%
  \end{picture}%
\endgroup%
\caption[A figure for the proof of Proposition~\ref{p.tvv}]{\small A figure for the proof of Proposition~\ref{p.tvv}. The dark shaded rectangles represent the approximate frequency supports of $\Delta_{\theta,j,\ell,k}$.}\label{f.esssup}
\end{figure}
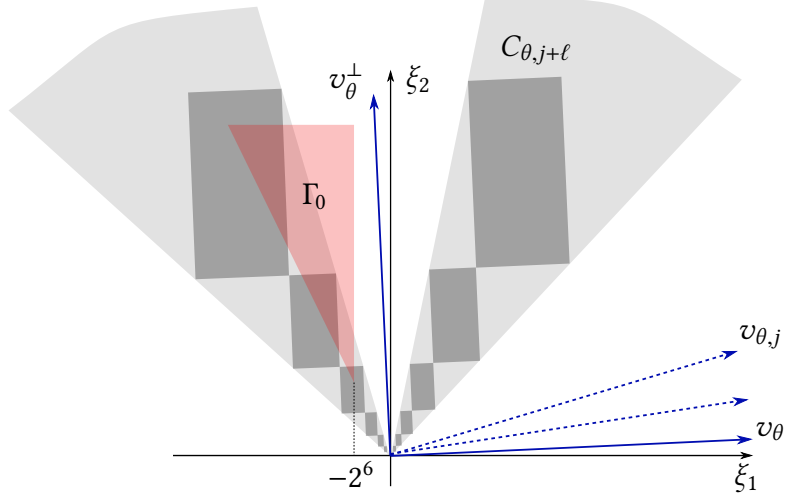
%%%%%%%%%%%%%%%%%%%%%%%%%%%%%% FIGURE FIGURE FIGURE

%%%%%%%%%%%%%%%%%%%%%%%%%%%%%% SECTION SECTION SECTION
\subsubsection{The initial decomposition}
With the  notation above we immediately get the identity
\[
H_{\mathsf v_\theta} f_\theta = \sum_{j\geq 1} \cic{1}_{E_{\theta, j}}  H_{v_{\theta,j}} f_{\theta};
\]
introducing without loss of generality the qualitative assumption that  $\inf_{\theta } \inf_{x \in \R^2 }\lambda_\theta>0$,  we obtain that  the sum in $j$ above is over a finite range $J \subset \mathbb N$.
Proceeding as in \cite[\S3]{GuoThiele}, we obtain the decomposition
\begin{equation}
\label{megadec}
H_{\mathsf{v}_\theta(x) } f(x) = \sum_{\ell\in \mathbb Z}  H_{v_{\theta,j}} \Delta_{\theta,j,\ell} f_{\theta} (x)  + 
\sum_{k\geq 5} \sum_{\tau \geq k+j+5} \Phi_{\theta,k}\widetilde{\Psi}_{\theta^\perp,\tau}f_\theta (x), \qquad x\in {E}_{\theta,j},
\end{equation}
with
\begin{equation}
\label{megadec1}
\Delta_{\theta,j,\ell}\coloneqq \sum_{k\geq 5} \Delta_{\theta,j,\ell,k}, \qquad 
\Delta_{\theta,j,\ell,k}\coloneqq 
\begin{cases}
B_{\theta^\perp,k+j+\lfloor \ell/2\rfloor} \Phi_{\theta,k}\widetilde{\Psi}_{\theta^\perp,k+j+\ell}, & \ell \leq -5, 
\\
\Phi_{\theta,k}\widetilde{\Psi}_{\theta^\perp,k+j+\ell}, & -4\leq \ell \leq 4,
\\  
A_{\theta^\perp,k+j+\lfloor \ell/2\rfloor}  \Phi_{\theta,k}\widetilde{\Psi}_{\theta^\perp,k+j+\ell},& \ell \geq 5.  \end{cases}
\end{equation}
Replacing the disjoint  support in $j$ of $\cic{1}_{E_{\theta,j}}$ by a square function in $j$ in the first summand of the right hand side of \eqref{megadec} we get
\begin{equation}
\begin{split}
\label{megadec2} &\quad 
 \|\{ H_{\mathsf v_\theta} f_\theta\}  \|_{L^p(\R^2;\, \ell^2(\v))}
\\ 
&\qquad  \leq \sum_{\ell\in \mathbb Z}
 \|\{ H_{v_{\theta,j}} \Delta_{\theta,j,\ell}    f_\theta\}  \|_{L^p(\R^2;\,\ell^2(\v\times J ))} + \Big\| \Big\{
\sum_{k\geq 5} \sum_{\tau \geq k+j_\theta(\cdot)+5} \Phi_{\theta,k}\widetilde{\Psi}_{\theta^\perp,\tau}   f_\theta \Big\} \Big\|_{L^p(\R^2;\,\ell^2(\v))}.
	\end{split}
\end{equation}

%%%%%%%%%%%%%%%%%%%%%%%%%%%%%% REMARK REMARK REMARK
\begin{remark} \label{LPremark} For the purposes of this section we may precompose all multipliers with a smooth Fourier restriction adopted to the cone $\Gamma_{0}$.  Therefore, making use of the smoothness and  vanishing of the Fourier transform of $\psi$, and relying upon relation \eqref{almostradial}, we obtain the symbol estimates
\[
\sup_{\theta\in \v} \Big(
\sup_{|\alpha|\leq 100}
|2^{-k}\xi|^{|\alpha|} |\partial_{\xi}^{\alpha}\widehat{  {\Psi}}_{\theta^\perp,k}(\xi)|  + |2^{-k}\xi|^{-1} |\widehat{  {\Psi}}_{\theta^\perp,k}(\xi)| \Big)\lesssim 1,\qquad k\in\mathbb Z.
\]
In particular, because of the uniformity of the H\"ormander-Mikhlin condition above in $\theta,$ the vector-valued multiplier operators
\[
\ell^q \ni\{f_\theta\}_\theta \mapsto \Big\{\sum_{k} \eps_{k} {\Psi}_{\theta^\perp,k} f_\theta \Big\}_{\theta}
\]
are bounded on $L^p(\R^2;\ell^q)$, $1<p,q<\infty$, uniformly over choices of $\eps_k\in \{-1,0,1\}$. A similar reasoning applies to $\widetilde{ {\Psi}}_{\theta^\perp,k}$ and (simpler) to $ { {\Phi}}_{\theta^\perp,k}$. In fact, in view of \eqref{almostradial} we could have worked with a radial frequency projection (with or without compact support) independent of $\theta$ in place of the directional projections along $\theta^\perp$. However, certain geometric considerations made in \cite{GuoThiele} and of use to us are more naturally expressed in terms of directional frequency projections of the type ${\Psi}_{\theta^\perp,k}$.
\end{remark}
%%%%%%%%%%%%%%%%%%%%%%%%%%%%%% REMARK REMARK REMARK

%%%%%%%%%%%%%%%%%%%%%%%%%%%%%% SECTION SECTION SECTION
\subsubsection{Estimation of the first summand in \eqref{megadec2}} In order to estimate the first summand, which is an error term, we will show that each term in the sum has sufficient decay with respect to the parameter $\ell$. We will need two tools, both of which are consequences of the $L^p$-boundedness of the directional maximal function $\M_{\v'}$ along $\v'\coloneqq   \{v_{\theta,j}\}_{\theta\in\v,j\in J}$, which holds since $\v'$ is a $(D+1)$-lacunary set; see~\eqref{e.direcmax}. The first is the C\'ordoba-Fefferman inequality, \cite{CorFef}, in the form given in \cite[Theorem 6.1]{DS},
\begin{equation} \label{CF1}
\|\{ H_{v_{\theta,j}} g_{\theta,j}\}   \|_{L^p(\R^2;\,\ell^2(\v\times J ))} \lesssim \|\{   g_{\theta,j}\}   \|_{L^p(\R^2;\,\ell^2(\v\times J ))} ,
\end{equation}
which we will apply to the doubly indexed sequence $g_{\theta,j}=\Delta_{\theta,j,\ell}    f_\theta$. Further, we claim the inequality\begin{equation}
\label{LPmulti}
\|\{ \Delta_{\theta,j,\ell} f_\theta \}   \|_{L^p(\R^2;\;\ell^2(\v\times J ))}    \lesssim  \omega_\ell \| \{  f_\theta \}   \|_{L^p(\R^2;\,\ell^2(\v))},
\end{equation} 
where $\omega_\ell$ is a suitable rapidly decaying sequence. We detail the proof in \S\ref{s.LPmulti} below but sketch the heuristic here. The operators $\Delta_{\theta,\ell,j}$ are approximate restrictions to the frequency cones
\begin{equation}
\label{cones2}
C_{\theta, \mu} \coloneqq \{\xi \in \R^2:\, 2^{-\mu-1}  |\xi \cdot \e^{2\pi i \theta^\perp}|<  |\xi \cdot \e^{2\pi i \theta}| < 2^{-\mu} |\xi \cdot  \e^{2\pi i \theta^\perp} |\}
\end{equation}
for $\mu=\ell+j$, so that, up to Schwartz tails
$ \Delta_{\theta,j,\ell} f_\theta     = \Delta_{\theta,j,\ell} f_{\theta,j,\ell}$ where $f_{\theta,j,\ell}$ denotes the rough frequency restriction of $f_{\theta}$ to $C_{\theta,j+\ell }$; see Figure~\ref{f.esssup}. We will see that the $ \Delta_{\theta,j,\ell}$ (and suitable Schwartz-tail free modifications) satisfy the vector valued bound
\begin{equation}
\label{vvDelta}
\left\|\{ \Delta_{\theta,j,\ell} g_{\theta,j} \}   \right\|_{L^p(\R^2;\,\ell^2(\v\times J ))}   
 \lesssim  \omega_\ell  \left\|\{   g_{\theta,j} \}   \right\|_{L^p(\R^2;\, \ell^2(\v\times J ))} 
\end{equation}
with the extra decay $\omega_\ell$, when $\mp\ell \geq  5$, due respectively to the smallness of the symbol of $\widetilde{\Psi}_{\theta^\perp,k+j+\ell}$ on the support of $B_{\theta^\perp,k+j+\lfloor \ell/2\rfloor}$, and to the mean zero of $\widetilde{\Psi}_{\theta^\perp,k+j+\ell}$ coupled with $A_{\theta^\perp,k+j+\lfloor \ell/2\rfloor}$ being constant near zero. Finally, \eqref{LPmulti} will follow from the vector valued bound for the multipliers associated to the cones $C_{\theta,j+\ell}$. Combining \eqref{LPmulti} with \eqref{CF1} yields the required summable decay for the first term in \eqref{megadec2}.

%%%%%%%%%%%%%%%%%%%%%%%%%%%%%% SECTION SECTION SECTION
\subsubsection{Estimation of the second summand in \eqref{megadec2}} We linearize the $L^p(\R^2;\ell^2)$-norm of the second summand by multiplying and integrating against some $\{g_\theta\}_{\theta \in \v}$ with $\|\{g_\theta\}\|_{L^{p'}(\R^2;\,\ell^2(\v))}=1$. We expand one of the convolutions in $\widetilde{\Psi}_{\theta^\perp,\tau}$ and rewrite the second summand as
\[
 \sum_{\theta\in\v} \int_{\R^2} \sum_{k\geq 5} \int_{\R} \sum_{\tau \geq k+j_\theta(x)+5} \Phi_{\theta,k}{\Psi}_{\theta^\perp,\tau}   f_\theta(x-t  \e^{2\pi i \theta^\perp}) \psi_{\tau} (t)g_\theta(x) \, \d t  \d x .
\]
Now we split the expression in the last display into two parts; the main term
\begin{equation}
\label{mainterm1}
 \sum_{\theta\in\v}\int_{\R^2}  \sum_{k\geq 5} \int_{\R} \sum_{\tau \geq k+j_\theta(x-t  \e^{2\pi i \theta^\perp})+5} \Phi_{\theta,k}{\Psi}_{\theta^\perp,\tau}   f_\theta(x-t  \e^{2\pi i \theta^\perp}) \psi_{\tau} (t)g_\theta(x) \, \d t  \d x ,
\end{equation}
and the error term which up to a sign is equal to
\begin{equation}
\label{mainterm2}
 \sum_{\theta\in\v}\int_{\R^2}  \sum_{k\geq 5} \int_{\R} \sum_{   j_{\theta} ^- (x,t)  <\tau -k -5 \leq j_{\theta} ^+ (x,t) } \Phi_{\theta,k}{\Psi}_{\theta^\perp,\tau}   f_\theta(x-t  \e^{2\pi i \theta^\perp}) \psi_{\tau} (t)g_\theta(x) \, \d t  \d x ,
\end{equation}
where $j_{\theta}^\pm(x,t)\coloneqq \pm\max\{\pm j_\theta(x), \pm j_\theta(x-t  \e^{2\pi i \theta^\perp})\}$. 

%%%%%%%%%%%%%%%%%%%%%%%%%%%%%% SECTION SECTION SECTION
\subsubsection{The main term coming from \eqref{mainterm1}}In the main term we may change variables in the $\d x\d t$ integral to pass the Littlewood-Paley operator  ${\Psi}_{\theta,\tau} $ to $g_\theta$ and later exchange summation order in $\tau,k$ to obtain
\[
\sum_{\theta\in\v}\int_{\R^2}
\sum_{\tau\geq 10+j_\theta(x)}   \sum_{5\leq k \leq \tau-j_\theta(x )-5} \Phi_{\theta,k}{\Psi}_{\theta^\perp,\tau}   f_\theta(x ) \Psi_{\theta^\perp,\tau} g_\theta(x) \,   \d x. 
\]
By $\ell^2(\v)$-valued Littlewood-Paley theory, see Remark \ref{LPremark}, 
\[
\left\|\{ \Psi_{\theta^\perp,\tau} g_\theta \}   \right\|_{L^{p'}(\R^2;\,\ell^2(\mathbb Z\times\v))} \lesssim  \left\|\{  g_\theta \}   \right\|_{L^{p'}(\R^2;\,\ell^2(\v) )}=1
\]
therefore it suffices to estimate 
\begin{equation} \label{fsfinal} \begin{split}
&\Big\|\Big\{ \sup_{r\geq 0}\Big|\sum_{5\leq k \leq 5+r} \Phi_{\theta,k}{\Psi}_{\theta^\perp,\tau}   f_\theta \Big| \Big\}   \Big\|_{L^{p}(\R^2;\,\ell^2 ( \mathbb Z\times \v ) )}  \lesssim \big\|\big\{  \M_{\e^{2\pi i \theta}}({\Psi}_{\theta^\perp,\tau}   f_\theta)  \big\}   \big\|_{L^{p}(\R^2;\,\ell^2  ( \mathbb Z\times\v ) )} 
\\ 
&\qquad \qquad\lesssim \left\|\left\{{\Psi}_{\theta^\perp,\tau}f_\theta \right\}\right\|_{L^{p}(\R^2;\,\ell^2  (\mathbb Z\times \v))} \lesssim \|\{f_\theta   \}   \|_{L^{p}(\R^2;\,\ell^2(\v))}  .
\end{split}
\end{equation}
The second estimate above follows by the Fefferman-Stein inequality of Lemma~\ref{l.FS} and the last estimate follows again by the vector-valued $L^p$-bounds of Remark~\ref{LPremark}. This concludes the treatment of the main  term.

%%%%%%%%%%%%%%%%%%%%%%%%%%%%%% SECTION SECTION SECTION
\subsubsection{The error term coming from \eqref{mainterm2}}   As noticed in \cite{GuoThiele}, the support conditions on $\psi_{\tau,\theta}$ and the Lipschitz condition on $\lambda_\theta$ yields that
\[
|\lambda_\theta(x-t\e^{2\pi i \theta}) -\lambda_\theta(x) | \leq |t| \leq 2^{-\tau} \leq 2^{-j_{\theta} ^-(x,t)-10}.
\]
It is then easy to see that $j_{\theta} ^+ (x,t)\leq j_{\theta} ^- (x,t)+1 $. Hence, by changing variables in the same fashion as for the main term, the expression in \eqref{mainterm2} can be written in the form
\[
\begin{split}
& \sum_{\theta\in\v}\int_{\R^2}\sum_{10+j^- _\theta <\tau\leq 10+j^+ _\theta } \bigg(\sum_{5\leq k \leq \tau-j_\theta ^- -5} \Phi_{\theta,k}{\Psi}_{\theta^\perp,\tau}   f_\theta(x )\bigg) \Psi_{\theta^\perp,\tau} g_\theta(x) \,   \d x
\\
& \qquad +
\sum_{\theta\in\v}\int_{\R^2}\sum_{10+j^+ _\theta  <\tau} \bigg( \sum_{\tau-j_\theta ^+ -5\leq k \leq \tau-j_\theta ^- -5} \Phi_{\theta,k}{\Psi}_{\theta^\perp,\tau}   f_\theta(x )\bigg) \Psi_{\theta^\perp,\tau} g_\theta(x) \,   \d x,
\end{split}
\]
where we suppressed the dependence of $j_\theta ^\pm$ on $x,t$. These two terms are estimated by an argument similar to the one used for the main term. In particular, the estimate for the first summand reduces to the estimate
\[
 \Big\|\Big\{  \sup_{0\leq r\leq 1}\Big|\sum_{5\leq k \leq 5+r} \Phi_{\theta,k}{\Psi}_{\theta^\perp,\tau}   f_\theta \Big| \Big\}   \Big\|_{L^{p}(\R^2;\,\ell^2 ( \mathbb Z\times \v ) )}   \lesssim \|\{f_\theta   \}   \|_{L^{p}(\R^2;\,\ell^2(\v))}  
\]
which is proven in the same fashion as~\eqref{fsfinal}. The second summand reduces in turn to the estimate
\[
 \Big\|\Big\{ \sup_{\substack{r>5\\0\leq s\leq 1}}\Big|\sum_{r\leq k \leq r+s} \Phi_{\theta,k}{\Psi}_{\theta^\perp,\tau}   f_\theta \Big| \Big\}   \Big\|_{L^{p}(\R^2;\,\ell^2 ( \mathbb Z\times \v ) )}   \lesssim \|\{f_\theta   \}   \|_{L^{p}(\R^2;\,\ell^2(\v))}  
\]
which is again estimated as in~\eqref{fsfinal}. This completes the treatment of the error term and the proof of Theorem~\ref{p.tvv}.

%%%%%%%%%%%%%%%%%%%%%%%%%%%%%% SECTION SECTION SECTION
\subsubsection{Proof of \eqref{LPmulti}}\label{s.LPmulti} We begin with a further decomposition of the multipliers $\Delta_{\theta,j,\ell}$ of \eqref{megadec1}. We write
\[
\Delta_{\theta,j,\ell} = \sum_{\beta\geq 0} \Delta^{\beta}_{\theta,j,\ell}
\]
where 
\[
\Delta^{0}_{\theta,j,\ell} \coloneqq \sum_{k\geq 5} A_{\theta^\perp,k+\ell+j}\Delta_{\theta,j,\ell,k}, \qquad  \Delta^{\beta}_{\theta,j,\ell} \coloneqq \sum_{k\geq 5} \Phi_{\theta^\perp,k+\ell+j+\beta} \Delta_{\theta,j,\ell,k}, \quad \beta\geq 1.
\]
If $C_{\theta,\mu}$ is as in \eqref{cones2} and $R_{\theta,\mu}$ is the corresponding Fourier restriction, it is now easy to verify that when $\beta\geq 1$
\[
\Delta^{\beta}_{\theta,j,\ell} = \Delta^{\beta}_{\theta,j,\ell} R_{\theta,\ell+j+\beta} .
\]
The first step is the observation that for all $\eps>0$ and  weights $u$ on $\R^2$
\begin{equation}
\label{onevalue}
\|\Delta^{\beta}_{\theta,j,\ell} g\|_{L^2( u)} \lesssim C_{\eps} 2^{-\beta} \omega_{\ell} \|  g\|_{L^2( w)}, \qquad w\coloneqq(\M_{\e^{2\pi i \theta}}^{\mathsf{st}} (u^{1+\eps}))^{\frac{1}{1+\eps}},
\end{equation}
where $\M_{v}^{\mathsf{st}}$ stands for the strong maximal function in the coordinates of $v$.
This stems from the observation that the dilation of the symbol $\Delta^{\beta}_{\theta,j,\ell} $ in direction $\e^{2\pi i \theta}$ by a factor of $2^{j+\ell+\beta}$ is a standard H\"ormander-Mikhlin multiplier  with symbol estimates controlled by $\omega_\ell$
and with additional decay in $\beta$ introduced  from the Schwartz tails of $\widetilde{\Psi}_{\theta^\perp,k+\ell+j}$ on the support of $\Phi_{\theta^\perp,k+\ell+j+\beta}$. Therefore such symbol satisfies the well known analogue of \eqref{onevalue} with the Hardy-Littlewood maximal function in place of $\M_{\e^{2\pi i \theta}}^{\mathsf{st}}$, namely the C\'ordoba-Fefferman inequality \cite{CorCFef}; see also \cite{Perez94}. The estimate \eqref{onevalue} readily yields the vector valued analogue 
\[
\|\{\Delta^{\beta}_{\theta,j,\ell} g_{\theta,j}\}\|_{L^2( u;\, \ell^{2}({\Theta \times J}))} \lesssim C_{\eps} 2^{-\beta} \omega_{\ell}  \| \{g_{\theta,j}\}\|_{L^2( w; \,\ell^{2}({\v \times J}))}, \qquad w\coloneqq(\M_{\v}^{\mathsf{st}} (u^{1+\eps}))^{\frac{1}{1+\eps}}.
\]
The $L^q$-boundedness of $\M_{\v}^{\mathsf{st}}$ for all $1<q<\infty$ and standard duality reasoning yield the inequality
\[
\|\{\Delta^{\beta}_{\theta,j,\ell} g_{\theta,j}\}\|_{L^p( \R^2;\, \ell^{2}({\v \times J}))} \lesssim  2^{-\beta} \omega_{\ell}  \| \{g_{\theta,j}\}\|_{L^p(\R^2;\,  \ell^{2}({\v \times J}))}, \quad 1<p<\infty
\]
which we apply to $g_{\theta,j}= R_{\theta,\ell+j+\beta} f_\theta$, thus obtaining 
\[
\begin{split}
&\quad \|\{\Delta^{\beta}_{\theta,j,\ell} f_{\theta,j}\}\|_{L^p(  \ell^{2}({\v \times J}))} = \|\{\Delta^{\beta}_{\theta,j,\ell} R_{\theta,\ell+j+\beta} f_\theta\}\|_{L^p(  \ell^{2}({\v \times J}))} 
\\ 
& \lesssim  2^{-\beta} \omega_{\ell}  \| \{R_{\theta,\ell+j+\beta} f_\theta\}\|_{L^p(   \ell^{2}({\v \times J}))}  \lesssim 2^{-\beta} \omega_{\ell} \| \{ f_\theta\}\|_{L^p(   \ell^{2}({\v}))}.
\end{split}
\] The last step is $\ell^2(\v)$-valued Littlewood-Paley theory of the (lacunary) cone multiplier operators $\{R_{\theta,\ell+j+\beta}:j\in J\}$. Finally, a summation over $\beta$ of the above inequality yields the claimed \eqref{LPmulti} and completes the proof.
 
%%%%%%%%%%%%%%%%%%%%%%%%%%%%%% SECTION SECTION SECTION
% \bib, bibdiv, biblist are defined by the amsrefs package.

\bibliographystyle{amsplain}
\end{document}